\documentclass[12pt,twoside]{amsart}
\usepackage{amsmath, amsthm, amscd, amsfonts, amssymb, graphicx}
\usepackage[bookmarksnumbered, plainpages]{hyperref}

\newtheorem{thm}{Theorem}[section]
\newtheorem{cor}[thm]{Corollary}

\numberwithin{equation}{section}

\begin{document}

\title{\bf Chen's inequalities for submanifolds in $(\kappa,\mu)$-contact space form with generalized semi-symmetric non-metric connections}
\author{Yong Wang}

\thanks{{\scriptsize
\hskip -0.4 true cm \textit{2010 Mathematics Subject Classification:}
53C40; 53C42.
 }}

\maketitle

\begin{abstract}
In this paper, we obtain Chen's inequalities for submanifolds in $(\kappa,\mu)$-contact space form with two kinds of generalized semi-symmetric non-metric connections.

\end{abstract}

\vskip 0.2 true cm


\pagestyle{myheadings}
\markboth{\rightline {\scriptsize Wang}}
         {\leftline{\scriptsize Chen inequalities for submanifolds in $(\kappa,\mu)$-contact space form}}

\bigskip
\bigskip


\section{ Introduction}

H. A. Hayden introduced the notion of a semi-symmetric metric connection on a
Riemannian manifold \cite{HA}. K. Yano studied a Riemannian manifold endowed with
a semi-symmetric metric connection \cite{YA}. Some properties of a Riemannian manifold
and a hypersurface of a Riemannian manifold with a semi-symmetric metric
connection were studied by T. Imai \cite{I1,I2}. Z. Nakao \cite{NA} studied submanifolds of
a Riemannian manifold with semi-symmetric metric connections. N. S. Agashe and
M. R. Chafle introduced the notion of a semisymmetric non-metric connection and
studied some of its properties and submanifolds of a Riemannian manifold with a
semi-symmetric non-metric connection \cite{AC1,AC2}. In \cite{QW}, we introduced a kind of generalized semi-symmetric non-metric connection
as a generalization of  a semi-symmetric metric connection and a semi-symmetric non-metric connection and we studied the Einstein warped products and multiply warped
products with this generalized semi-symmetric non-metric connection. In \cite{LHZ}, Li, He and Zhao introduced anothor kind of generalized semi-symmetric non-metric connection.
They studied submanifolds in a Riemannian manifold with this kind of generalized semi-symmetric non-metric connection.

On the other hand, one of the basic problems in submanifold theory is to find
simple relationships between the extrinsic and intrinsic invariants of a submanifold.
B. Y. Chen \cite{BYC0,BYC1,BYC2,BYC3,BYC4} established inequalities in this respect, called Chen inequalities.
Afterwards, many geometers studied similar problems for different submanifolds in
various ambient spaces; see, for example, \cite{ALV,DDVV,DS1,GKKT,JW1,JW2,JZ,LLV,LSV,ZZS}.
 In \cite{MO1,OM}, Mihai and $\ddot{{\rm O}}$zg$\ddot{{\rm u}}$r studied Chen inequalities for submanifolds of real
space forms with a semi-symmetric metric connection and a semi-symmetric non-metric connection, respectively.
Later, in \cite{ZZS}, Zhang, Zhang and Song obtained Chen's inequalities for submanifolds of a
Riemannian manifold of quasi-constant curvature endowed with a semi-symmetric
metric connection.  At the meantime, in \cite{ZPZ}, Zhang {\it et al.} obtained Chen's inequalities for submanifolds of a
Riemannian manifold of nearly quasi-constant curvature endowed with a semi-symmetric
non-metric connection. In \cite{MO2}, Mihai and $\ddot{{\rm O}}$zg$\ddot{{\rm u}}$r studied Chen inequalities for submanifolds of complex space forms and Sasakian space forms with a semi-symmetric metric connection.
In \cite{AHTZ}, Ahmad, He, Tang and Zhao obtained Chen's inequalities for submanifolds in $(\kappa,\mu)$-contact space form with a semi-symmetric metric connection. In \cite{ASL}, Ahmad, Shahzad and Li obtained Chen's inequalities for submanifolds in $(\kappa,\mu)$-contact space form with a semi-symmetric non-metric connection.

The paper is organized as follows.
In Section 2, we give a brief introduction about two kinds of generalized semi-symmetric non-metric connections.
In Section 3, we establish Chen's inequalities for submanifolds in $(\kappa,\mu)$-contact space form with the first generalized semi-symmetric non-metric connection.
In Section 4, we establish Chen's inequalities for submanifolds in $(\kappa,\mu)$-contact space form with the second generalized semi-symmetric non-metric connection.


\vskip 1 true cm

\section{ Preliminaries}

Let $N^{n+p}$ be an $(n+p)$-dimensional Riemannian manifold with Riemannian metric $g$ and the Levi-Civita connection $\widehat{\overline{\nabla}}$.
Let ${\overline{\nabla}}$ be a linear connection defined by
\begin{equation} \label{P1}
\overline{\nabla}_XY=\widehat{\overline{\nabla}}_XY+\lambda_1\pi(Y)X-\lambda_2g(X,Y)P,
\end{equation}
where $P$ is a vector field on $N$ and $\pi(X)=g(P,X).$ We call $\overline{\nabla}$ as the first generalized semi-symmetric non-metric connection.
When $\lambda_1=\lambda_2=1$, $\overline{\nabla}$ is the semi symmetric metric connection and when
$\lambda_1=1$, $\lambda_2=0$, $\overline{\nabla}$ is the semi symmetric non-metric connection. Let
$$
\overline{R}(X,Y)Z:=\overline{\nabla}_X\overline{\nabla}_YZ-\overline{\nabla}_Y\overline{\nabla}_XZ-\overline{\nabla}_{[X,Y]}Z
$$
 be the curvature tensor of $\overline{\nabla}$. The curvature tensor of $\widehat{\overline{\nabla}}$, say $\widehat{\overline{R}}$, can be defined as the same.
Let
\begin{equation*}
\alpha(X,Y)=(\widehat{\overline{\nabla}}_X\pi)(Y)-\lambda_1\pi(X)\pi(Y)+\frac{\lambda_2}{2}g(X,Y)\pi(P),
\end{equation*}
and
\begin{equation*}
\beta(X,Y)=\frac{\pi(P)}{2}g(X,Y)+\pi(X)\pi(Y)
\end{equation*}
 be two $(0,2)$-tensors, then we have
\begin{align} \label{P1}
\overline{R}(X,Y,Z,W)&=\widehat{\overline{R}}(X,Y,Z,W)+\lambda_1\alpha(X,Z)g(Y,W)-\lambda_1\alpha(Y,Z)g(X,W)\notag\\
&~+\lambda_2g(X,Z)\alpha(Y,W)-\lambda_2g(Y,Z)\alpha(X,W)+\lambda_2(\lambda_1-\lambda_2)g(X,Z)\beta(Y,W)\\
&~-\lambda_2(\lambda_1-\lambda_2)g(Y,Z)\beta(X,W).\notag
\end{align}

For simplicity, we denote by ${\rm tr}(\alpha)=\lambda$ and ${\rm tr}(\beta)=\mu$.

Let $M^n$ be an $n$-dimensional submanifold of an $(n+p)$-dimensional Riemannian manifold $N^{n+p}$. On the submanifold $M$
we consider the induced generalized semi-symmetric non-metric connection denoted by $\nabla$ and the induced Levi-Civita connection denoted by $\widehat{\nabla}$. Let $R$ and $\widehat{R}$ be the curvature tensors
of $\nabla$ and $\widehat{\nabla}$. Decomposing the vector field $P$ on $N$ uniquely into its tangent and normal components $P^\top$ and $P^\bot$, respectively, then we have $P=P^\top+P^\bot$.
The Gauss formulas with respect to $\nabla$ and $\widehat{\nabla}$ can be written as:
\begin{equation*}
\overline{\nabla}_XY=\nabla_XY+h(X,Y),~~~~~X,Y\in \Gamma(TM),
\end{equation*}
\begin{equation*}
\widehat{\overline{\nabla}}_XY=\widehat{\nabla}_XY+\widehat{h}(X,Y),~~~~~X,Y\in \Gamma(TM),
\end{equation*}
where $\widehat{h}$ is the second fundamental form of $M$ in $N$ and
\begin{equation*}
h(X,Y)=\widehat{h}(X,Y)-\lambda_2g(X,Y)P^\bot.
\end{equation*}
In $N^{n+p}$ we can choose a local orthonormal frame $e_1,\cdots,e_n,e_{n+1},\cdots,e_{n+p},$ such that, restricting to $M$, $e_1,\cdots,e_n$ are tangent to $M^n$. We write $h_{ij}^r=g(h(e_i,e_j),e_r)$.
The squared length of $h$ is $||h||^2=\sum_{i,j=1}^ng(h(e_i,e_j),h(e_i,e_j))$ and the mean curvature vector of $M$ associated to $\nabla$ is $H=\frac{1}{n}\sum_{i=1}^nh(e_i,e_i).$ Similarly, the
mean curvature vector of $M$ associated to $\widehat{\nabla}$ is $\widehat{H}=\frac{1}{n}\sum_{i=1}^n\widehat{h}(e_i,e_i).$
We have the Gauss equation
\begin{align} \label{P2}
\overline{R}(X,Y,Z,W)&=R(X,Y,Z,W)-g(h(X,W),h(Y,Z))+g(h(Y,W),h(X,Z))\notag\\
&+(\lambda_1-\lambda_2)g(h(Y,Z),P)g(X,W)+(\lambda_2-\lambda_1)g(h(X,Z),P)g(Y,W).
\end{align}
\indent Let $\Pi\subset T_xM$, $x\in M$, be a $2$-plane section and $\Pi={\rm span}\{e_1,e_2\}$ where $e_1,e_2$ are orthonormal basis of $\Pi$. Since $R(X,Y,Z,W)\neq -R(X,Y,W,Z)$, we define $K(\Pi)$ the sectional curvature of $M$ with the induced connection $\nabla$ as follows:
\begin{equation} \label{P3}
K(\Pi)=\frac{1}{2}[R(e_1,e_2,e_2.e_1)-R(e_1,e_2,e_1.e_2)],
\end{equation}
then $K(\Pi)$ is independent of the choice of $e_1,e_2$.
For any orthonormal basis $\{e_1,\cdots,e_n\}$ of the tangent space
$T_xM$, the scalar curvature $\tau$ at $x$ is defined by
\begin{equation}\label{P4}
\tau(x)=\sum_{1\leq i<j\leq n}K(e_i\wedge e_j)=\frac{1}{2}\sum_{1\leq i,j\leq n}R(e_i,e_j,e_j,e_i).
\end{equation}

Let ${\overline{\widetilde{\nabla}}}$ be a linear connection called the second generalized semi-symmetric non-metric connection defined by
\begin{equation} \label{P1}
\overline{\widetilde{\nabla}}_XY=\widehat{\overline{\nabla}}_XY+a\pi(X)Y+b\pi(Y)X.
\end{equation}
 Let
$
\overline{\widetilde{R}}$
be the curvature tensor of $\overline{\widetilde{\nabla}}$.

Let $(\widehat{\overline{\nabla}}_X\pi)(Y)=\alpha'(X,Y)$ be a $(0,2)$-tensor.
The following equality states the relationship between $\overline{\widetilde{R}}$ and $\widehat{\overline{R}}$ (see \cite{LHZ} (7)):

\begin{align} \label{P2}
\overline{\widetilde{R}}(X,Y,Z,W)&=\widehat{\overline{R}}(X,Y,Z,W)-a\alpha'(Y,X)g(Z,W)+a\alpha'(X,Y)g(Z,W)\\
&-b\alpha'(Y,Z)g(X,W)+b\alpha'(X,Z)g(Y,W)\notag\\
&+b^2\pi(Y)\pi(Z)g(X,W)-b^2\pi(X)\pi(Z)g(Y,W)
.\notag
\end{align}

We denote by ${\rm tr}(\alpha')=\lambda'$.
The Gauss formulas with respect to $\widetilde{\nabla}$ can be written as:
\begin{equation*}
\overline{\widetilde{\nabla}}_XY=\widetilde{\nabla}_XY+\widetilde{h}(X,Y),~~~~~X,Y\in \Gamma(TM).
\end{equation*}
Then
\begin{equation*}
\widetilde{h}(X,Y)=\widehat{h}(X,Y).
\end{equation*}
By (25) in \cite{LHZ}, we have the Gauss equation
\begin{align} \label{P2}
\overline{\widetilde{R}}(X,Y,Z,W)&=\widetilde{R}(X,Y,Z,W)-g(\widetilde{h}(X,W),\widetilde{h}(Y,Z))+g(\widetilde{h}(Y,W),\widetilde{h}(X,Z))\notag\\
&+b\pi(\widetilde{h}(Y,Z))g(X,W)-b\pi(\widetilde{h}(X,Z))g(Y,W).
\end{align}

\vskip 1 true cm

\section{Chen inequalities for submanifolds in $(\kappa,\mu)$-contact space form with the first generalized semi-symmetric non-metric connection.}

 \quad A $(2m+1)$-dimensional Riemannian manifold $(N^{2m+1},g)$ has an almost contact metric structure if it admits a $(1,1)$-tensor field $\varphi$, a vector field $\xi$ and $1$-form $\eta$ satisfying:\\
 $$\varphi^2X=-X+\eta(X)\xi,~~\eta(\xi)=1,~~\varphi\xi=0,~~\eta\circ\varphi=0$$
 $$g(X,\varphi Y)=-g(\varphi X,Y),~~g(X,\xi)=\eta(X),$$
 for any vector fields $X,Y$ on $TN$. Let $\Phi$ denote the fundamental $2$-form in $N$, given by $\Phi(X,Y)=g(X,\varphi Y)$, for all $X,Y$ on $TN$. If $\Phi=d\eta$, then $N$ is called {\it a contact metric manifold}.
 In a contact metric manifold $N$, the $(1,1)$-tensor field $h'$ defined by $2h'=\textit{L}_\xi\varphi$ is symmetric and satisfies
 $$h'\xi=0,~~h'\varphi+\varphi h'=0,~~ \widehat{\overline{\nabla}}\xi=-\varphi-\varphi h',~~ {\rm trace}(h')= {\rm trace}(\varphi h')=0.$$
 The $(\kappa,\mu)$-nullity distribution of a contact metric manifold $N$ is a distribution
 $$N(\kappa,\mu):p\rightarrow N_p(\kappa,\mu)=\{Z\in T_pN|\widehat{\overline{R}}(X,Y)Z=\kappa[g(Y,Z)X-g(X,Z)Y]$$
 $$+\mu[g(Y,Z)h'X-g(X,Z)h'Y]\}$$
 where $\kappa$ and $\mu$ are constants. If $\xi\in N(\kappa,\mu)$, $N$ is called a $(\kappa,\mu)$-contact metric manifold
 A plane section $\Pi$ in $T_xN$ is called a {\it $\varphi$-section} if it is spanned by $X$ and $\varphi X$, where $X$ is a unit tangent vector field orthogonal to $\xi$. The sectional curvature of a
$\varphi$-section is called a {\it $\varphi$-sectional curvature.} If  a $(\kappa,\mu)$-contact metric manifold $N$ has constant $\varphi$-section curvature $c$, then it is called a
$(\kappa,\mu)$-contact space form and it is denoted by $N(c)$.
The curvature tensor $\widehat{\overline{R}}$ of $N(c)$ with respect to the Levi-civita connection $\widehat{\overline{\nabla}}$ is expressed by
\begin{align}\widehat{\overline{R}}(X,Y)Z&=\frac{c+3}{4}[g(Y,Z)X-g(X,Z)Y]+\frac{c+3-4\kappa}{4}[\eta(X)\eta(Z)Y-\eta(Y)\eta(Z)X\\
&+g(X,Z)\eta(Y)\xi-g(Y,Z)\eta(X)\xi]+\frac{c-1}{4}[2g(X,\varphi Y)\varphi Z+g(X,\varphi Z)\varphi Y\notag\\
&-g(Y,\varphi Z)\varphi X]+\frac{1}{2}[g(h'Y,Z)h'X-g(h'X,Z)h'Y+g(\varphi h'X,Z)\varphi h'Y\notag\\
&-g(\varphi h'Y,Z)\varphi h'X]-g(X,Z)h'Y+g(Y,Z)h'X+\eta(X)\eta(Z)h'Y-\eta(Y)\eta(Z)h'X\notag\\
&-g(h'X,Z)Y+g(h'Y,Z)X-g(h'Y,Z)\eta(X)\xi+g(h'X,Z)\eta(Y)\xi\notag\\
&+\mu[\eta(Y)\eta(Z)h'X-\eta(X)\eta(Z)h'Y+g(h'Y,Z)\eta(X)\xi-g(h'X,Z)\eta(Y)\xi]
,\notag
\end{align}
for vector fields $X,Y,Z$ on $N(c)$.
Let $M$ be an $n$-dimensional submanifold of a $(2m+1)$-dimensional $(\kappa,\mu)$-contact space form of constant $\varphi$-sectional curvature $N(c)$. For any tangent vector field $X$ to $M$,
we put $$\varphi X=\hat{P}X+FX,$$
where $\hat{P}X$ and $FX$ are tangential and normal components of $\varphi X$ respectively and we decompose $$\xi=\xi^\top+\xi^\bot,$$
where $\xi^\top$ and $\xi^\bot$ denote the tangential and normal parts of $\xi$. Let $(\varphi h')^\top X$ and $(h')^\top X$ denote the tangential parts of $\varphi h'X$ and $h'X$ respectively. We set
$$\vartheta^2=\sum_{i,j=1}^ng(e_i,\vartheta e_j)^2,~~~~~~\vartheta\in\{\widehat{P},(\varphi h')^\top,(h')^\top\}.$$
We denote by $\beta(\Pi)=g^2(\hat{P}e_1,e_2)=g^2(\varphi e_1,e_2)$, where $\{e_1,e_2\}$ is an orthonormal basis of a $2$-plane section $\Pi$. Put
$$\gamma(\Pi)=(\eta(e_1))^2+(\eta(e_2))^2,$$
$$\theta(\Pi)=\eta(e_1)^2g((h')^\top e_2,e_2)+\eta(e_2)^2g((h')^\top e_1,e_1)-2\eta(e_1)\eta(e_2)g((h')^\top e_1,e_2).$$
Then $\gamma(\Pi)$ and $\theta(\Pi)$ do not depend on the choice of orthonormal basis $\{e_1,e_2\}$.
We have the following Chen's inequalities:\\

\begin{thm}
 Let $M^n$, $n\geq 3$, be an $n$-dimensional submanifold of a $(2m+1)$-dimensional $(\kappa,\mu)$-contact space form of constant $\varphi$-sectional curvature $N(c)$ endowed with a connection
$\overline{\nabla}$, then
\begin{align*}
\tau(x)-K(\Pi)&\leq \frac{(n+1)(n-2)}{2}\frac{c+3}{4}+\frac{c+3-4\kappa}{4}[-(n-1)||\xi^\top||^2+\gamma(\Pi)]\\
&+\frac{c-1}{8}[3||\widehat{P}||^2-6\beta(\Pi)]
+\frac{1}{4}[||(\varphi h')^\top||^2-||(h')^\top||^2\\
&+|{\rm trace}((h')^\top)|^2-|{\rm trace}((\varphi h')^\top)|^2]
-\frac{1}{2}[{\rm det}(h'|_{\Pi})-{\rm det}((\varphi h')|_{\Pi})]\\
&+(n-1){\rm tr}((h')^\top)-{\rm tr}(h'|_{\Pi})+(1-\mu)[g(\xi^\top,h'\xi^\top)-{\rm tr}((h')^\top)||\xi^\top||^2\\
&+\theta(\Pi)]-\frac{\lambda_1+\lambda_2}{2}(n-1)\lambda-\frac{\lambda_2}{2}(\lambda_1-\lambda_2)(n-1)\mu\\
&-\frac{\lambda_1-\lambda_2}{2}(n-1)n\pi(H)+\frac{\lambda_1+\lambda_2}{2}{\rm tr}(\alpha|_\Pi)\\
&+\frac{\lambda_2(\lambda_1-\lambda_2)}{2}{\rm tr}(\beta|_\Pi)
+\frac{\lambda_1-\lambda_2}{2}g({\rm tr}(h|_\Pi),P)\\
&+\frac{n^2(n-2)}{2(n-1)}\|H\|^2.
\end{align*}
\end{thm}

\begin{proof}
Let $x\in M$ and $\{e_1,\cdots,e_n\}$ and $\{e_{n+1},\cdots,e_{n+p}\}$ be orthonormal basis of $T_xM$ and $T_x^\bot M$ respectively. For $X=W=e_i$, $Y=Z=e_j$, $i\neq j$
by (2.2),(2.3) and (3.1), we have
\begin{align}
{R}(e_i,e_j,e_j,e_i)&=\frac{c+3}{4}+\frac{c+3-4\kappa}{4}[-\eta(e_i)^2-\eta(e_j)^2]\\
&+\frac{c-1}{4}[3g(e_i,\varphi e_j)^2]+\frac{1}{2}[g(e_i,\varphi h'e_j)^2-g(e_i, h'e_j)^2\notag\\
&+g(e_i, h'e_i)g(e_j, h'e_j)-g(e_i,\varphi h'e_i)g(e_j,\varphi h'e_j)]
+g(e_i, h'e_i)\notag\\
&+2\eta(e_i)\eta(e_j)g(e_i, h'e_j)-g(e_i, h'e_i)\eta(e_j)^2-g(e_j, h'e_j)\eta(e_i)^2+g(e_j, h'e_j)\notag\\
&+\mu[g(e_i, h'e_i)\eta(e_j)^2+g(e_j, h'e_j)\eta(e_i)^2-2\eta(e_i)\eta(e_j)g(e_i, h'e_j)]\notag\\
&-\lambda_1\alpha(e_j,e_j)-\lambda_2\alpha(e_i,e_i)-\lambda_2(\lambda_1-\lambda_2)\beta(e_i,e_i)\notag\\
&-g(h(e_i,e_j),h(e_i,e_j))+g(h(e_i,e_i),h(e_j,e_j))-(\lambda_1-\lambda_2)g(h(e_j,e_j),P).\notag
\end{align}
Then
\begin{align} \label{s}
\tau(x)&=\frac{1}{2}\sum_{1\leq i\neq j\leq n}{R}(e_i,e_j,e_j,e_i)\\
&=
\frac{(n-1)n}{2}\frac{c+3}{4}+\frac{3(c-1)}{8}||\widehat{P}||^2-\frac{c+3-4\kappa}{4}(n-1)||\xi^\top||^2\notag\\
&+\frac{1}{4}[||(\varphi h')^\top||^2-||(h')^\top||^2+|{\rm trace}((h')^\top)|^2-|{\rm trace}((\varphi h')^\top)|^2]\notag\\
&+(n-1){\rm tr}((h')^\top)+(1-\mu)[g(\xi^\top,h'\xi^\top)-{\rm tr}((h')^\top)||\xi^\top||^2]\notag\\
&-\frac{\lambda_1+\lambda_2}{2}(n-1)\lambda-\frac{\lambda_2}{2}(\lambda_1-\lambda_2)(n-1)\mu\notag\\
&~~-\frac{\lambda_1-\lambda_2}{2}(n-1)n\pi(H)+\sum_{r=n+1}^{n+p}\sum_{1\leq i<j\leq n}[h^r_{ii}h^r_{jj}-(h^r_{ij})^2].\notag
\end{align}
Let $\Pi={\rm span}\{ e_1,e_2\}$ and by (3.2), we have
\begin{align} \label{C3}
&{R}(e_1,e_2,e_2,e_1)=\frac{c+3}{4}-\frac{c+3-4\kappa}{4}\gamma(\Pi)\\
&+\frac{3(c-1)}{4}\beta(\Pi)
+\frac{1}{2}[{\rm det}(h'|_{\Pi})-{\rm det}((\varphi h')|_{\Pi})]\notag\\
&+{\rm tr}(h'|_{\Pi})+(\mu-1)\theta(\Pi)-\lambda_1\alpha(e_2,e_2)\notag\\
&-\lambda_2\alpha(e_1,e_1)
-\lambda_2(\lambda_1-\lambda_2)\beta(e_1,e_1)\notag\\
&-(\lambda_1-\lambda_2)g(h(e_2,e_2),P)+\sum_{r=n+1}^{n+p}[h^r_{11}h^r_{22}-(h^r_{12})^2].\notag
\end{align}
Similarly, we have
\begin{align} \label{C4}
{R}(e_1,e_2,e_1,e_2)&=-\{\frac{c+3}{4}-\frac{c+3-4\kappa}{4}\gamma(\Pi)+\frac{3(c-1)}{4}\beta(\Pi)\\
&+\frac{1}{2}[{\rm det}(h'|_{\Pi})-{\rm det}((\varphi h')|_{\Pi})]+{\rm tr}(h'|_{\Pi})+(\mu-1)\theta(\Pi)\}\notag\\
&+\lambda_1\alpha(e_1,e_1)+\lambda_2\alpha(e_2,e_2)+\lambda_2(\lambda_1-\lambda_2)\beta(e_2,e_2)\notag\\
&+(\lambda_1-\lambda_2)g(h(e_1,e_1),P)-\sum_{r=n+1}^{n+p}[h^r_{11}h^r_{22}-(h^r_{12})^2].\notag
\end{align}
By (2.4),(3.4) and (3.5), we have
\begin{align} \label{C5}
K(\Pi)&=\frac{c+3}{4}-\frac{c+3-4\kappa}{4}\gamma(\Pi)+\frac{3(c-1)}{4}\beta(\Pi)\\
&+\frac{1}{2}[{\rm det}(h'|_{\Pi})-{\rm det}((\varphi h')|_{\Pi})]+{\rm tr}(h'|_{\Pi})+(\mu-1)\theta(\Pi)\notag\\
&-\frac{\lambda_1+\lambda_2}{2}{\rm tr}(\alpha|_\Pi)
-\frac{\lambda_2(\lambda_1-\lambda_2)}{2}{\rm tr}(\beta|_\Pi)\notag\\
&-\frac{\lambda_1-\lambda_2}{2}g({\rm tr}(h|_\Pi),P)+\sum_{r=n+1}^{n+p}[h^r_{11}h^r_{22}-(h^r_{12})^2].\notag
\end{align}
By Lemma 2.4 in \cite{ZZS}, we get
\begin{equation} \label{C4}
\sum_{r=n+1}^{n+p}[\sum_{1\leq i<j\leq n}h^r_{ii}h^r_{jj}-h^r_{11}h^r_{22}-\sum_{1\leq i<j\leq n}(h^r_{ij})^2+(h^r_{12})^2]\leq \frac{n^2(n-2)}{2(n-1)}\|H\|^2.
\end{equation}
By (3.3),(3.6) and (3.7) we complete the proof.

\end{proof}

\indent Similar to Corollaries 3.4 in \cite{ZZS}, we have\\

\begin{cor}  If $P$ is a tangent vector field on $M$, then $h=\widehat{h}$. In this case, the equality case of Theorem 3.1 holds at a point $x\in M$ if and only if,
with respect to a suitable orthonormal basis $\{e_A\}$ at $x$, the shape operators $A_r=A_{e_r}$ take the following forms:
$$A_{n+1}=\left(
  \begin{array}{ccccc}
    h_{11}^{n+1}&  0 & 0 & \cdots & 0 \\
   0&   h_{22}^{n+1} & 0 & \cdots & 0 \\
    0 & 0 &  h_{11}^{n+1}+h_{22}^{n+1} & \cdots & 0 \\
    \vdots & \vdots & \vdots & \ddots & 0 \\
    0&0 & 0 & \cdots & h_{11}^{n+1}+h_{22}^{n+1} \\
  \end{array}
\right)
$$
and
$$A_{r}=\left(
  \begin{array}{ccccc}
    h_{11}^{r}&  h_{12}^{r} & 0 & \cdots & 0 \\
    h_{12}^{r}&  -h_{11}^{r} & 0 & \cdots & 0 \\
    0 & 0 & 0 & \cdots & 0 \\
    \vdots & \vdots & \vdots & \ddots & 0 \\
    0&0 & 0 & \cdots & 0\\
  \end{array}
\right),r=n+2,\cdots,n+p.
$$
\end{cor}

\begin{thm}  Let $M^n$, $n\geq 3$, be an $n$-dimensional submanifold of a $(2m+1)$-dimensional $(\kappa,\mu)$-contact space form of constant $\varphi$-sectional curvature $N(c)$ endowed with a connection
$\overline{\nabla}$, then:

\indent {\rm (i)} For each unit vector $X$ in $T_xM$ we have
\begin{align} \label{C7}
{\rm Ric}(X)&\leq \frac{(n-1)(c+3)}{4}+\frac{3(c-1)}{4}||\widehat{P}X||^2-\frac{c+3-4\kappa}{4}[(n-2)\eta(X)^2+||\xi^\top||^2]\\
&+\frac{1}{2}[||(\varphi h'X)^\top||^2-||(h'X)^\top||^2+g(X,h'X){\rm tr}((h')^\top)-g(X,\varphi h'X){\rm tr}((\varphi h')^\top)]\notag\\
&+(n-2-||\xi^\top||^2+\mu ||\xi^\top||^2)g(X,hX)+(1-\eta(X)^2+\mu\eta(X)^2){\rm tr}((h')^\top)\notag\\
&+(2-2\mu)\eta(X)g(X,h'(\xi^\top))-\lambda_1\lambda+[\lambda_1-\lambda_2(n-1)]\alpha(X,X)\notag\\
&-\lambda_2(\lambda_1-\lambda_2)(n-1)\beta(X,X)
-n(\lambda_1-\lambda_2)\pi (H)+(\lambda_1-\lambda_2)\pi(h(X,X))
+\frac{n^2}{4}\|H\|^2.\notag
\end{align}
\indent {\rm (ii)} If $H(x)=0$, then a unit tangent vector $X$ at $x$ satisfies the equality case of \eqref{C7} if and only if $X\in \mathcal{N}(x)=\{X\in T_xM|h(X,Y)=0,\forall Y\in T_xM\}.$\\
\indent {\rm (iii)} The equality of \eqref{C7} holds for all unit tangent vector at $x$ if and only if \\
\indent $~~ h^r_{ij}=0,~~ i,j=1,2,\cdots,n, r=n+1,\cdots,n+p.$\\
\end{thm}

\begin{proof}
Let $X\in T_xM$ be a unit tangent vector at $x$. We choose the orthonormal basis $\{e_1,\cdots, e_n\}$ such that $e_1=X$. Then by (3.2), we have
\begin{align} \label{C8}
{\rm Ric}(X)&=\sum_{j=2}^n{R}(e_1,e_j,e_j,e_1)=\frac{(n-1)(c+3)}{4}+\frac{3(c-1)}{4}||\widehat{P}X||^2\\
&-\frac{c+3-4\kappa}{4}[(n-2)\eta(X)^2+||\xi^\top||^2]+\frac{1}{2}[||(\varphi h'X)^\top||^2\notag\\
&-||(h'X)^\top||^2+g(X,h'X){\rm tr}((h')^\top)-g(X,\varphi h'X){\rm tr}((\varphi h')^\top)]\notag\\
&+(n-2-||\xi^\top||^2+\mu ||\xi^\top||^2)g(X,hX)+(1-\eta(X)^2+\mu\eta(X)^2){\rm tr}((h')^\top)\notag\\
&+(2-2\mu)\eta(X)g(X,h'(\xi^\top))-\lambda_1\lambda+[\lambda_1-\lambda_2(n-1)]\alpha(X,X)\notag\\
&-\lambda_2(\lambda_1-\lambda_2)(n-1)\beta(X,X)
-n(\lambda_1-\lambda_2)\pi (H)+(\lambda_1-\lambda_2)\pi(h(X,X))\notag\\
&+\sum_{r=n+1}^{n+p}\sum_{j=2}^n[h^r_{11}h^r_{jj}-(h_{1j}^r)^2].\notag
\end{align}
By Lemma 2.5 in \cite{ZZS}, we get
\begin{equation*}
\sum_{r=n+1}^{n+p}\sum_{j=2}^nh^r_{11}h^r_{jj}\leq \frac{n^2}{4}\|H\|^2,
\end{equation*}
which together with \eqref{C8} gives \eqref{C7}. The Proofs of (ii) and (iii) are same as the proofs of (ii) and (iii) in Theorem 5.2 in \cite{ZZS}.
\end{proof}

Let $L$ be a $k$-plane section of $T_xM$, $x\in M$ and $X$ a unit vector in $L$. We choose an orthonormal frame $e_1,\cdots,e_k$ of $L$ such that $e_1=X$. In \cite{BYC1}, Chen defined the $k$-Ricci curvature
of $L$ at $X$ by
$${\rm Ric}_L(X)=K_{12}+K_{13}+\cdots+K_{1k}.$$
For an integer $k$, $2\leq k\leq n$, the Riemannian invariant $\Theta_k$ on $M$ defined by
$$\Theta_k(x)=\frac{1}{k-1}{\rm inf}\{{\rm Ric}_L(X)|L,X\},~~x\in M,$$
where $L$ runs over all $k$-plane sections in $T_xM$ and $X$ runs over all unit vectors in $L$. We have
\begin{equation} \label{C9}
\tau(x)\geq \frac{n(n-1)}{2}\Theta_k(x).
\end{equation}

\begin{thm}  Let $M^n$, $n\geq 3$, be an $n$-dimensional submanifold of a $(2m+1)$-dimensional $(\kappa,\mu)$-contact space form of constant $\varphi$-sectional curvature $N(c)$ endowed with a connection
$\overline{\nabla}$, then: for any integer $k$, $2\leq k\leq n$, and any point $x\in M$, we have
\begin{align}
n(n-1)\|H\|^2(x)&\geq n(n-1)\Theta_k(x)-\frac{c+3}{4}{(n-1)n}-\frac{3(c-1)}{4}||\widehat{P}||^2+\frac{c+3-4\kappa}{2}(n-1)||\xi^\top||^2\\
&-\frac{1}{2}[||(\varphi h')^\top||^2-||(h')^\top||^2+|{\rm trace}((h')^\top)|^2-|{\rm trace}((\varphi h')^\top)|^2]\notag\\
&-2(n-1){\rm tr}((h')^\top)-2(1-\mu)[g(\xi^\top,h'\xi^\top)-{\rm tr}((h')^\top)||\xi^\top||^2]\notag\\
&+{\lambda}(n-1)(\lambda_1+\lambda_2)+\lambda_2(\lambda_1-\lambda_2)(n-1){\mu}+(n-1)n(\lambda_1-\lambda_2)\pi(H).\notag
\end{align}
\end{thm}

\begin{proof}
 We choose the orthonormal frame $\{e_1,\cdots,e_n,e_{n+1},\cdots,e_{n+p}\}$ at $x$ such that the $e_{n+1}$ is in the direction of the mean curvature vector $H(x)$ and
$\{e_1,\cdots,e_n\}$ diagonalize the shape operator $A_{n+1}$. Then the shape operators take the following form
\begin{equation} \label{C10}
A_{n+1}={\rm diag}(a_1,a_2,\cdots,a_n), ~~{\rm trace}A_r=0,~~r=n+2,\cdots,n+p.
\end{equation}
By (3.3), we have
\begin{align}
&2\tau-\frac{c+3}{4}{(n-1)n}-\frac{3(c-1)}{4}||\widehat{P}||^2+\frac{c+3-4\kappa}{2}(n-1)||\xi^\top||^2\\
&-\frac{1}{2}[||(\varphi h')^\top||^2-||(h')^\top||^2+|{\rm trace}((h')^\top)|^2-|{\rm trace}((\varphi h')^\top)|^2]\notag\\
&-2(n-1){\rm tr}((h')^\top)-2(1-\mu)[g(\xi^\top,h'\xi^\top)-{\rm tr}((h')^\top)||\xi^\top||^2]\notag\\
&+{\lambda}(n-1)(\lambda_1+\lambda_2)+\lambda_2(\lambda_1-\lambda_2)(n-1){\mu}+(n-1)n(\lambda_1-\lambda_2)\pi(H)\notag\\
&=n^2||H||^2-||h||^2.\notag
\end{align}
By \eqref{C10}, we get
$$\|h\|^2=\sum_{j=1}^{n}a_j^2+\sum_{r=n+2}^{n+p}\sum_{1\leq i,j\leq l}(h^r_{ij})^2.$$
Since $\sum_{j=1}^{n}a_j^2\geq n\|H\|^2$,  by (3.10) and (3.13), we completes the proof.
\end{proof}

Nextly, we prove Casorati inequalities. We denote the Casorati curvature with respect to $\overline{\nabla}$ by
$$\mathcal{C}=\frac{1}{n}\sum_{r=n+1}^{n+p}\sum_{i,j=1}^n(h_{ij}^r)^2.$$
Suppose now that $L$ is an $l$-dimensional subspace of $T_xM$, $l\geq 2$, and $\{e_1,\cdots,e_l\}$ be an orthonormal basis of $L$. Then the Casorati curvature of the $l$-plane section $L$ with
respect to $\overline{\nabla}$ is given by
$$\mathcal{C}(L)=\frac{1}{l}\sum_{r=n+1}^{n+p}\sum_{i,j=1}^l(h_{ij}^r)^2.$$
Then normalized $\delta$-Casorati curvature $\delta_c(n-1)$ and $\hat{\delta}_c(n-1)$ with respect to $\overline{\nabla}$ are given by
$$[\delta_c(n-1)]_x=\frac{1}{2}\mathcal{C}_x+\frac{n+1}{2n}{\rm inf}\{\mathcal{C}(L)|L~{\rm~a~~hyperplane~~ of~~} T_xM\},$$
$$[\hat{\delta}_c(n-1)]_x={2}\mathcal{C}_x-\frac{2n-1}{2n}{\rm sup}\{\mathcal{C}(L)|L~{\rm~a~~hyperplane~~ of~~} T_xM\}.$$

\begin{thm}  Let $M^n$, $n\geq 3$, be an $n$-dimensional submanifold of a $(2m+1)$-dimensional $(\kappa,\mu)$-contact space form of constant $\varphi$-sectional curvature $N(c)$ endowed with a connection
$\overline{\nabla}$, then:

\indent {\rm (i)} The normalized $\delta$-Casorati curvature $\delta_c(n-1)$ satisfies
\begin{align} 
2\tau\leq& n(n-1)\delta_c(n-1)+\frac{c+3}{4}{(n-1)n}+\frac{3(c-1)}{4}||\widehat{P}||^2-\frac{c+3-4\kappa}{2}(n-1)||\xi^\top||^2\\
&+\frac{1}{2}[||(\varphi h')^\top||^2-||(h')^\top||^2+|{\rm trace}((h')^\top)|^2-|{\rm trace}((\varphi h')^\top)|^2]\notag\\
&+2(n-1){\rm tr}((h')^\top)+2(1-\mu)[g(\xi^\top,h'\xi^\top)-{\rm tr}((h')^\top)||\xi^\top||^2]\notag\\
&-{\lambda}(n-1)(\lambda_1+\lambda_2)-\lambda_2(\lambda_1-\lambda_2)(n-1){\mu}-(n-1)n(\lambda_1-\lambda_2)\pi(H).
\notag
\end{align}
Moreover, if $P$ is tangent to $M$, the equality case of (3.14) holds if and only if $M$ is an invariantly quasi-umbilical submanifold in $N$, such that with respect to suitable orthonormal tangent
frame $\{e_1,\cdots,e_n\}$ and normal orthonormal frame $\{e_{n+1},\cdots,e_{n+p}\}$, the shape operator $A_r=A_{e_r}$, $r\in\{n+1,\cdots,n+p\}$,taking the following forms:
$$A_{n+1}={\rm diag}(a,a,\cdots,a,2a),~~A_{n+2}=\cdots=A_{n+p}=0.$$
\indent {\rm (ii)} The normalized $\delta$-Casorati curvature $\hat{\delta}_c(n-1)$ satisfies
\begin{align}
2\tau\leq& n(n-1)\hat{\delta}_c(n-1)+\frac{c+3}{4}{(n-1)n}+\frac{3(c-1)}{4}||\widehat{P}||^2-\frac{c+3-4\kappa}{2}(n-1)||\xi^\top||^2\\
&+\frac{1}{2}[||(\varphi h')^\top||^2-||(h')^\top||^2+|{\rm trace}((h')^\top)|^2-|{\rm trace}((\varphi h')^\top)|^2]\notag\\
&+2(n-1){\rm tr}((h')^\top)+2(1-\mu)[g(\xi^\top,h'\xi^\top)-{\rm tr}((h')^\top)||\xi^\top||^2]\notag\\
&-{\lambda}(n-1)(\lambda_1+\lambda_2)-\lambda_2(\lambda_1-\lambda_2)(n-1){\mu}-(n-1)n(\lambda_1-\lambda_2)\pi(H).
\notag
\end{align}
Moreover, if $P$ is tangent to $M$, the equality case of (3.14) holds if and only if $M$ is an invariantly quasi-umbilical submanifold in $N$, such that with respect to suitable orthonormal tangent
frame $\{e_1,\cdots,e_n\}$ and normal orthonormal frame $\{e_{n+1},\cdots,e_{n+p}\}$, the shape operator $A_r=A_{e_r}$, $r\in\{n+1,\cdots,n+p\}$,taking the following forms:
$$A_{n+1}={\rm diag}(2a,2a,\cdots,2a,a),~~A_{n+2}=\cdots=A_{n+p}=0.$$
\end{thm}

\begin{proof} By (3.3), we have
\begin{align}
2\tau=&\frac{c+3}{4}{(n-1)n}+\frac{3(c-1)}{4}||\widehat{P}||^2-\frac{c+3-4\kappa}{2}(n-1)||\xi^\top||^2\\
&+\frac{1}{2}[||(\varphi h')^\top||^2-||(h')^\top||^2+|{\rm trace}((h')^\top)|^2-|{\rm trace}((\varphi h')^\top)|^2]\notag\\
&+2(n-1){\rm tr}((h')^\top)+2(1-\mu)[g(\xi^\top,h'\xi^\top)-{\rm tr}((h')^\top)||\xi^\top||^2]\notag\\
&-{\lambda}(n-1)(\lambda_1+\lambda_2)-\lambda_2(\lambda_1-\lambda_2)(n-1){\mu}\notag\\
&-(n-1)n(\lambda_1-\lambda_2)\pi(H)+n^2||H||^2-n\mathcal{C}.\notag
\notag
\end{align}

Let 
\begin{align}
Q=&\frac{1}{2}n(n-1)\mathcal{C}+\frac{(n-1)(n+1)}{2}\mathcal{C}(L)
-2\tau+\frac{c+3}{4}{(n-1)n}\\
&+\frac{3(c-1)}{4}||\widehat{P}||^2-\frac{c+3-4\kappa}{2}(n-1)||\xi^\top||^2\notag\\
&+\frac{1}{2}[||(\varphi h')^\top||^2-||(h')^\top||^2+|{\rm trace}((h')^\top)|^2-|{\rm trace}((\varphi h')^\top)|^2]\notag\\
&+2(n-1){\rm tr}((h')^\top)+2(1-\mu)[g(\xi^\top,h'\xi^\top)-{\rm tr}((h')^\top)||\xi^\top||^2]\notag\\
&-{\lambda}(n-1)(\lambda_1+\lambda_2)-\lambda_2(\lambda_1-\lambda_2)(n-1){\mu}\notag\\
&-(n-1)n(\lambda_1-\lambda_2)\pi(H).\notag
\notag
\end{align}
By Lemma 4 in \cite{ZZ}, we have $Q\geq 0$, so we get (3.14). Similar to the proof of Theorem 3 in \cite{ZZ}, we can prove Theorem 3.5.
\end{proof}

\section{Chen inequalities for submanifolds in $(\kappa,\mu)$-contact space form with the second generalized semi-symmetric non-metric connection.}

Let $\widetilde{\tau}$ and $\widetilde{K}(\Pi)$ be the scalar curvature and the section curvature with respect to the connection ${\widetilde{\nabla}}$ respectively.
 Similar to Theorem 3.1, we have
\begin{thm}
 Let $M^n$, $n\geq 3$, be an $n$-dimensional submanifold of a $(2m+1)$-dimensional $(\kappa,\mu)$-contact space form of constant $\varphi$-sectional curvature $N(c)$ endowed with a connection $\overline{\widetilde{\nabla}}$
, then
\begin{align*}
\widetilde{\tau}(x)-\widetilde{K}(\Pi)&\leq \frac{(n+1)(n-2)}{2}\frac{c+3}{4}+\frac{c+3-4\kappa}{4}[-(n-1)||\xi^\top||^2+\gamma(\Pi)]\\
&+\frac{c-1}{8}[3||\widehat{P}||^2-6\beta(\Pi)]
+\frac{1}{4}[||(\varphi h')^\top||^2-||(h')^\top||^2\\
&+|{\rm trace}((h')^\top)|^2-|{\rm trace}((\varphi h')^\top)|^2]
-\frac{1}{2}[{\rm det}(h'|_{\Pi})-{\rm det}((\varphi h')|_{\Pi})]\\
&+(n-1){\rm tr}((h')^\top)-{\rm tr}(h'|_{\Pi})+(1-\mu)[g(\xi^\top,h'\xi^\top)-{\rm tr}((h')^\top)||\xi^\top||^2\\
&+\theta(\Pi)]-\frac{n-1}{2}b\lambda'+\frac{n-1}{2}b^2||P^\top||^2\\
&-\frac{b}{2}(n-1)n\pi(\widetilde{H})+\frac{b}{2}{\rm tr}(\alpha'|_{\Pi})-\frac{b^2}{2}||P_{\Pi}||^2\\
&+\frac{b}{2}\pi({\rm tr}(\widetilde{h}|_{\Pi}))
+\frac{n^2(n-2)}{2(n-1)}\|\widetilde{H}\|^2.
\end{align*}
\end{thm}

\begin{proof}
Let $x\in M$ and $\{e_1,\cdots,e_n\}$ and $\{e_{n+1},\cdots,e_{n+p}\}$ be orthonormal basis of $T_xM$ and $T_x^\bot M$ respectively. For $X=W=e_i$, $Y=Z=e_j$, $i\neq j$
by (2.7),(2.8) and (3.1), we have
\begin{align}
\widetilde{R}(e_i,e_j,e_j,e_i)&=\frac{c+3}{4}+\frac{c+3-4\kappa}{4}[-\eta(e_i)^2-\eta(e_j)^2]\\
&+\frac{c-1}{4}[3g(e_i,\varphi e_j)^2]+\frac{1}{2}[g(e_i,\varphi h'e_j)^2-g(e_i, h'e_j)^2\notag\\
&+g(e_i, h'e_i)g(e_j, h'e_j)-g(e_i,\varphi h'e_i)g(e_j,\varphi h'e_j)]
+g(e_i, h'e_i)\notag\\
&+2\eta(e_i)\eta(e_j)g(e_i, h'e_j)-g(e_i, h'e_i)\eta(e_j)^2-g(e_j, h'e_j)\eta(e_i)^2+g(e_j, h'e_j)\notag\\
&+\mu[g(e_i, h'e_i)\eta(e_j)^2+g(e_j, h'e_j)\eta(e_i)^2-2\eta(e_i)\eta(e_j)g(e_i, h'e_j)]\notag\\
&-b\alpha'(e_j,e_j)+b^2\pi(e_j)^2-b\pi(\widetilde{h}(e_j,e_j))+\sum_{r=n+1}^{n+p}[\widetilde{h}^r_{ii}\widetilde{h}^r_{jj}-(\widetilde{h}^r_{ij})^2].\notag
\end{align}
Then
\begin{align} \label{s}
\widetilde{\tau}(x)&=
\frac{(n-1)n}{2}\frac{c+3}{4}+\frac{3(c-1)}{8}||\widehat{P}||^2-\frac{c+3-4\kappa}{4}(n-1)||\xi^\top||^2\\
&+\frac{1}{4}[||(\varphi h')^\top||^2-||(h')^\top||^2+|{\rm trace}((h')^\top)|^2-|{\rm trace}((\varphi h')^\top)|^2]\notag\\
&+(n-1){\rm tr}((h')^\top)+(1-\mu)[g(\xi^\top,h'\xi^\top)-{\rm tr}((h')^\top)||\xi^\top||^2]\notag\\
&-\frac{n-1}{2}b\lambda'+\frac{n-1}{2}b^2||P^\top||^2-\frac{b}{2}(n-1)n\pi(\widetilde{H})\notag\\
&+\sum_{r=n+1}^{n+p}\sum_{1\leq i<j\leq n}[\widetilde{h}^r_{ii}\widetilde{h}^r_{jj}-(\widetilde{h}^r_{ij})^2]
.\notag
\end{align}
Let $\Pi={\rm span}\{ e_1,e_2\}$ and by (4.1), we have
\begin{align} \label{C3}
&\widetilde{R}(e_1,e_2,e_2,e_1)=\frac{c+3}{4}-\frac{c+3-4\kappa}{4}\gamma(\Pi)\\
&+\frac{3(c-1)}{4}\beta(\Pi)
+\frac{1}{2}[{\rm det}(h'|_{\Pi})-{\rm det}((\varphi h')|_{\Pi})]\notag\\
&+{\rm tr}(h'|_{\Pi})+(\mu-1)\theta(\Pi)-b\alpha'(e_2,e_2)+b^2\pi(e_2)^2\notag\\
&-b\pi(\widetilde{h}(e_2,e_2))+\sum_{r=n+1}^{n+p}[\widetilde{h}^r_{11}\widetilde{h}^r_{22}-(\widetilde{h}^r_{12})^2]\notag
\end{align}
Similarly, we have
\begin{align} \label{C4}
\widetilde{R}(e_1,e_2,e_1,e_2)&=-\{\frac{c+3}{4}-\frac{c+3-4\kappa}{4}\gamma(\Pi)+\frac{3(c-1)}{4}\beta(\Pi)\\
&+\frac{1}{2}[{\rm det}(h'|_{\Pi})-{\rm det}((\varphi h')|_{\Pi})]+{\rm tr}(h'|_{\Pi})+(\mu-1)\theta(\Pi)\}\notag\\
&+b\alpha'(e_1,e_1)-b^2\pi(e_1)^2+b\pi(\widetilde{h}(e_1,e_1))\notag\\
&-\sum_{r=n+1}^{n+p}[\widetilde{h}^r_{11}\widetilde{h}^r_{22}-(\widetilde{h}^r_{12})^2]\notag
\end{align}
By (2.4),(4.3) and (4.4), we have
\begin{align} \label{C5}
\widetilde{K}(\Pi)&=\frac{c+3}{4}-\frac{c+3-4\kappa}{4}\gamma(\Pi)+\frac{3(c-1)}{4}\beta(\Pi)\\
&+\frac{1}{2}[{\rm det}(h'|_{\Pi})-{\rm det}((\varphi h')|_{\Pi})]+{\rm tr}(h'|_{\Pi})+(\mu-1)\theta(\Pi)\notag\\
&-\frac{b}{2}{\rm tr}(\alpha'|_{\Pi})+\frac{b^2}{2}||P_{\Pi}||^2
-\frac{b}{2}\pi({\rm tr}(\widetilde{h}|_{\Pi}))\notag\\
&+\sum_{r=n+1}^{n+p}[\widetilde{h}^r_{11}\widetilde{h}^r_{22}-(\widetilde{h}^r_{12})^2].\notag
\end{align}
By (4.2),(4.5) and (3.7) we complete the proof.
\end{proof}

Similarly to Theorem 3.3, we have

\begin{thm}  Let $M^n$, $n\geq 3$, be an $n$-dimensional submanifold of a $(2m+1)$-dimensional $(\kappa,\mu)$-contact space form of constant $\varphi$-sectional curvature $N(c)$ endowed with a connection
$\overline{\widetilde{\nabla}}$, then: for each unit vector $X$ in $T_xM$ we have
\begin{align} 
\widetilde{{\rm Ric}}(X)&\leq \frac{(n-1)(c+3)}{4}+\frac{3(c-1)}{4}||\widehat{P}X||^2-\frac{c+3-4\kappa}{4}[(n-2)\eta(X)^2+||\xi^\top||^2]\\
&+\frac{1}{2}[||(\varphi h'X)^\top||^2-||(h'X)^\top||^2+g(X,h'X){\rm tr}((h')^\top)-g(X,\varphi h'X){\rm tr}((\varphi h')^\top)]\notag\\
&+(n-2-||\xi^\top||^2+\mu ||\xi^\top||^2)g(X,hX)+(1-\eta(X)^2+\mu\eta(X)^2){\rm tr}((h')^\top)\notag\\
&+(2-2\mu)\eta(X)g(X,h'(\xi^\top))-b\lambda+b\alpha'(X,X)+b^2||P^\top||^2\notag\\
&-b^2\pi(X)^2-nb\pi(\widetilde{H})+b\pi(\widetilde{h}(X,X))+
\frac{n^2}{4}\|\widetilde{H}\|^2.\notag
\end{align}
\end{thm}

Similarly to Theorem 3.4, we have

\begin{thm}  Let $M^n$, $n\geq 3$, be an $n$-dimensional submanifold of a $(2m+1)$-dimensional $(\kappa,\mu)$-contact space form of constant $\varphi$-sectional curvature $N(c)$ endowed with a connection
$\overline{\widetilde{\nabla}}$, then: for any integer $k$, $2\leq k\leq n$, and any point $x\in M$, we have
\begin{align}
n(n-1)\|\widetilde{H}\|^2(x)&\geq n(n-1)\widetilde{\Theta}_k(x)-\frac{c+3}{4}{(n-1)n}-\frac{3(c-1)}{4}||\widehat{P}||^2+\frac{c+3-4\kappa}{2}(n-1)||\xi^\top||^2\\
&-\frac{1}{2}[||(\varphi h')^\top||^2-||(h')^\top||^2+|{\rm trace}((h')^\top)|^2-|{\rm trace}((\varphi h')^\top)|^2]\notag\\
&-2(n-1){\rm tr}((h')^\top)-2(1-\mu)[g(\xi^\top,h'\xi^\top)-{\rm tr}((h')^\top)||\xi^\top||^2]\notag\\
&+(n-1)b\lambda'-(n-1)b^2||P^\top||^2+b(n-1)n\pi(\widetilde{H}).\notag
\end{align}
\end{thm}

We denote the Casorati curvature with respect to $\overline{\widetilde{\nabla}}$ by
$\widetilde{\mathcal{C}}$. We denote the normalized $\delta$-Casorati curvature by $\widetilde{\delta}_c(n-1)$ and $\widetilde{\hat{\delta}}_c(n-1)$ with respect to $\overline{\widetilde{\nabla}}$.
Similarly to Theorem 3.5, we have
\begin{thm}  Let $M^n$, $n\geq 3$, be an $n$-dimensional submanifold of a $(2m+1)$-dimensional $(\kappa,\mu)$-contact space form of constant $\varphi$-sectional curvature $N(c)$ endowed with a connection
$\overline{\widetilde{\nabla}}$, then:

\indent {\rm (i)} The normalized $\delta$-Casorati curvature $\widetilde{\delta}_c(n-1)$ satisfies
\begin{align}
2\widetilde{\tau}\leq& n(n-1)\widetilde{\delta}_c(n-1)+\frac{c+3}{4}{(n-1)n}+\frac{3(c-1)}{4}||\widehat{P}||^2-\frac{c+3-4\kappa}{2}(n-1)||\xi^\top||^2\\
&+\frac{1}{2}[||(\varphi h')^\top||^2-||(h')^\top||^2+|{\rm trace}((h')^\top)|^2-|{\rm trace}((\varphi h')^\top)|^2]\notag\\
&+2(n-1){\rm tr}((h')^\top)+2(1-\mu)[g(\xi^\top,h'\xi^\top)-{\rm tr}((h')^\top)||\xi^\top||^2]\notag\\
&-(n-1)b\lambda'+(n-1)b^2||P^\top||^2-b(n-1)n\pi(\widetilde{H}).\notag
\end{align}
\indent {\rm (ii)} The normalized $\delta$-Casorati curvature $\widetilde{\hat{\delta}}_c(n-1)$ satisfies
\begin{align}
2\widetilde{\tau}\leq& n(n-1)\widetilde{\hat{\delta}}_c(n-1)+\frac{c+3}{4}{(n-1)n}+\frac{3(c-1)}{4}||\widehat{P}||^2-\frac{c+3-4\kappa}{2}(n-1)||\xi^\top||^2\\
&+\frac{1}{2}[||(\varphi h')^\top||^2-||(h')^\top||^2+|{\rm trace}((h')^\top)|^2-|{\rm trace}((\varphi h')^\top)|^2]\notag\\
&+2(n-1){\rm tr}((h')^\top)+2(1-\mu)[g(\xi^\top,h'\xi^\top)-{\rm tr}((h')^\top)||\xi^\top||^2]\notag\\
&-(n-1)b\lambda'+(n-1)b^2||P^\top||^2-b(n-1)n\pi(\widetilde{H}).\notag
\end{align}
\end{thm}

\vskip 1 true cm

\section{Acknowledgements}

The author was supported in part by  NSFC No.11771070.

\vskip 1 true cm


\bigskip
\bigskip

\noindent {\footnotesize {\it Y. Wang} \\
{School of Mathematics and Statistics, Northeast Normal University, Changchun 130024, China}\\
{Email: wangy581@nenu.edu.cn}


\begin{thebibliography}{20}

\bibitem{AC1}
N.S. Agashe and M.R. Chafle, {\it A semi-symmetric non-metric connection on a Riemannian manifold},
Indian J. Pure Appl. Math. {\bf23} (1992), 399-409.

\bibitem{AC2}
N.S. Agashe and M.R. Chafle, {\it On submanifolds of a Riemannian manifold with a semi-symmetric non-metric connection},
Tensor {\bf55} (1994), 120-130.

\bibitem{AHTZ}A. Ahmad, G. He, W. Tang and P. Zhao, {\it Chen's inequalities for submanifolds in $(\kappa,\mu)$-contact space form with a semi-symmetric metric connection}, Open Math. {\bf 16} (2018), 380-391.

\bibitem{ASL}A. Ahmad, F. Shahzad and J. Li, {\it Chen's inequalities for submanifolds in $(\kappa,\mu)$-contact space form with a semi-symmetric non-metric connection}, J. Appl. Math. Phys., {bf 6} (2018), 389-404.

\bibitem{ALV}
M. Aquib M, J.W. Lee, G.E. V\^{\i}lcu, D.W. Yoon, {\it Classification of Casorati ideal Lagrangian submanifolds in complex space forms},
Differ.Geom. Appl. {\bf63} (2019), 30-49.


\bibitem{BYC0}
B.Y. Chen, {\it Some pinching and classification theorems for minimal submanifolds},
Arch. Math. (Basel) {\bf60} (1993), 568-578.

\bibitem{BYC1}
B.Y. Chen, {\it Relations between Ricci curvature and shape operator for submanifolds with arbitrary codimensions},
Glasg. Math. J. {\bf41} (1999), 33-41.

\bibitem{BYC2}
B.Y. Chen, {\it Mean curvature and shape operator of isometric immersions in real space forms},
Glasg. Math. J. {\bf 38} (1996), 87-97.

\bibitem{BYC3}
B.Y. Chen, {\it A Riemannian invariant and its applications to Einstein manifolds},
Bull. Austral. Math. Soc. {\bf 70} (2004), 55-65.


\bibitem{BYC4}
B.Y. Chen, {\it A general optimal inequality for warped products in complex projective spaces and its applications},
Proc. Japan Acad. Ser. A Math. Sci. {\bf79} (2002), 96-100.

\bibitem{DDVV}
P.J. De Smet, F. Dillen, L. Verstraelen and L. Vrancken, {\it A pointwise inequality in submanifold theory},
Arch. Math. (Brno) {\bf 35} (1999), 115-128.

\bibitem{DS1}
S. Decu, M. Petrovie-Torgasev, A. Sebekovic and L. Verstraelen, {\it Ricci and Casorati principal directions of Wintgen ideal submanifolds},
Filomat {\bf28} (2014), 657-661.

\bibitem{JZ}
J. Ge and Z. Tang, {\it A proof of the DDVV conjecture and its equality case},
Pacific J. Math. {\bf 237} (2008), 87-95.

\bibitem{GKKT}
M. G\"{u}lbahar, E. K{\i}l{\i}\c{c}, S. Kele\c{s} and M.M. Tripathi, {\it Some basic inequalities for submanifolds of nearly quasi-constant curvature manifolds},
Differ. Geom. Dyn. Syst. {\bf 16} (2014), 156-167.

\bibitem{HA}
H.A. Hayden, {\it Subspaces of a space with torsion},
Proc. London Math. Soc. {\bf34} (1932), 27-50.

\bibitem{I1}
T. Imai, {\it Hypersurfaces of a Riemannian manifold with semi-symmetric metric connection},
Tensor {\bf23} (1972), 300-306.

\bibitem{I2}
T. Imai, {\it Notes on semi-symmetric metric connections},
Tensor {\bf 24} (1972), 293-296.

\bibitem{JW1}
C.W. Lee, J.W. Lee, {\it Inequalities on Sasakian Statistical Manifolds in Terms of Casorati Curvatures},
Mathematics {\bf 6} (2018), 259.

\bibitem{LLV}
C.W. Lee, J.W. Lee, G.E. V\^{\i}lcu, {\it Optimal inequalities for the normalized $\delta$-Casorati curvatures of submanifolds in Kenmotsu space forms},
Advances in Geometry {\bf 17} (2017), 355-362.

\bibitem{JW2}
J.W. Lee and G.E. V\^{\i}lcu, {\it Inequalities for generalized normalized $\delta$-Casorati curvatures of slant submanifolds in quaternionic space forms},
Taiwan. J. Math. {\bf 19} (2015), 691-702.

\bibitem{LHZ}
J. Li, G. He, P. Zhao, {\it On submanifolds in a Riemannian manifold with a semi-symmetric non-metric connection}, Symmetry {\bf 9 }(2017), no. 7, Paper No. 112, 10 pp.

\bibitem{LSV}
M.A. Lone, M.H. Shahid, G.E. V\^{\i}lcu, {\it On Casorati curvatures of submanifolds in pointwise Kenmotsu space forms},
Mathematical Physics, Analysis and Geometry {\bf 22} (2019), 2.


\bibitem{MO1}
A. Mihai and C. $\ddot{{\rm O}}$zg$\ddot{{\rm u}}$r, {\it Chen inequalities for submanifolds of real space forms with a semi-symmetric metric connection},
Taiwanese J. Math. {\bf 14} (2010), 1465-1477.

\bibitem{MO2}
A. Mihai and C. $\ddot{{\rm O}}$zg$\ddot{{\rm u}}$r, {\it Chen inequalities for submanifolds of complex space forms and Sasakian space forms endowed with semi-symmetric metric connections},
 Rocky Mountain J. Math. {\bf 41} (2011), no. 5, 1653-1673.

\bibitem{NA}
Z. Nakao, {\it Submanifolds of a Riemannian manifold with semisymmetric metric connections}.
Proc. Amer. Math. Soc. {\bf54} (1976), 261-266.


\bibitem{OM}
C. $\ddot{{\rm O}}$zg$\ddot{{\rm u}}$r, A, Mihai, {\it Chen inequalities for submanifolds of real space forms with a semi-symmetric non-metric connection},
Canad. Math. Bull. {\bf55} (2012), 611-622.

\bibitem{P}
K.S. Park, {\it Inequalities for the Casorati curvatures of real hypersurfaces in some Grassmannians},
Taiwanese J. Math. {\bf 22} (2018), 63-77.

\bibitem{QW}
Q. Qu, Y. Wang, {\it Multiply warped products with a quarter-symmetric connection},
J. Math. Anal. Appl. {\bf431} (2015), 955-987.

\bibitem{TGKK}
M.M. Tripathi, M. G\"{u}lbahar, E. K{\i}l{\i}\c{c}, S. Kele\c{s}, {\it Inequalities for scalar curvature of pseudo-Riemannian submanifolds},
J. Geome. Phys.  {\bf112} (2017), 74-84.

\bibitem{V}
G.E. V\^{\i}lcu, {\it An optimal inequality for Lagrangian submanifolds in complex space forms involving Casorati curvature},
J. Math. Anal. Appl. {\bf465} (2018), 1209-1222.

\bibitem{YA}
K. Yano, {\it On semi-symmetric metric connection}.
Rev. Roumaine Math. Pures Appl. {\bf15} (1970), 1579-1586.

\bibitem{Y}
H. Y{\i}ld{\i}r{\i}m, L. Vrancken, {\it $\delta^{\sharp}(2, 2) $-Ideal centroaffine hypersurfaces of dimension $5$},
Taiwanese J. Math. {\bf21} (2017), 283-304.

\bibitem{ZPZ}
P. Zhang, X. Pan, L. Zhang, {\it Inequalities for submanifolds of a Riemannian manifold of nearly quasi-constant curvature with a semi-symmetric non-metric connection},
Rev. Un. Mat. Argentina {\bf56} (2015), 1-19.

\bibitem{ZZ}P. Zhang, L. Zhang, {\it Casorati inequalities for submanifolds in a Riemannian manifold of quasi-constant curvature with a semi-symmetric metric connection}, Symmetry {\bf 8} (2016), no. 4, Art. 19, 9 pp.

\bibitem{ZZS}
P. Zhang, L. Zhang, W. Song, {\it Chen's inequalities for submanifolds of a Riemannian manifold of quasi-constant curvature with a semi-symmetric metric connection},
Taiwanese J. Math. {\bf18} (2014), 1841-1862.






\end{thebibliography}
\end{document}